\documentclass{amsart}

\usepackage{amsmath}
\usepackage{amsthm}
\usepackage{amsfonts}
\usepackage{amssymb}

\newcommand\R{{\mathbb{R}}}
\newcommand\C{{\mathbb{C}}}

\renewcommand\P{{\mathbf{P}}}
\newcommand\E{{\mathbf{E}}}

\renewcommand\Im{{\operatorname{Im}}}
\renewcommand\Re{{\operatorname{Re}}}
\renewcommand\th{{\operatorname{th}}}
\newcommand\eps{{\varepsilon}}

\newcommand\tr{\operatorname{trace}}

\newcommand\condo{{{\bf C0}}}
\newcommand\condone{{{\bf C1}}}

\theoremstyle{plain}
  \newtheorem{theorem}{Theorem}

  \newtheorem{proposition}[theorem]{Proposition}
  
  \newtheorem{lemma}[theorem]{Lemma}
  \newtheorem{corollary}[theorem]{Corollary}

\theoremstyle{definition}
  \newtheorem{definition}[theorem]{Definition}
  
  \newtheorem{remark}[theorem]{Remark}

\include{psfig}

\begin{document}

\title{Random matrices: Universal properties  of eigenvectors}

\author{Terence Tao}
\address{Department of Mathematics, UCLA, Los Angeles CA 90095-1555}
\email{tao@math.ucla.edu}
\thanks{T. Tao is supported by NSF Research Award DMS-0649473, the NSF Waterman award and a grant from the MacArthur Foundation.}

\author{Van Vu}
\address{Department of Mathematics, Rutgers, New Jersey, NJ 08854}
\email{vanvu@math.rutgers.edu}
\thanks{V. Vu is supported by research grants DMS-0901216 and AFOSAR-FA-9550-09-1-0167.}

\begin{abstract}   The four moment theorem asserts, roughly speaking, that the joint distribution of a small number of eigenvalues of a Wigner random matrix (when measured at the scale of the mean eigenvalue spacing) depends only on the first four moments of the entries of the matrix.  In this paper, we extend the four moment theorem to also cover the coefficients of the \emph{eigenvectors} of a Wigner random matrix.  A similar result (with different hypotheses) has been proved recently by Knowles and Yin, using a different method. 

 As an application, we prove some central limit theorems for these eigenvectors. In another application, we prove  a universality result for the resolvent, up to the real axis. This implies universality of the inverse matrix. 
\end{abstract}

\maketitle

%\setcounter{tocdepth}{2}
%\tableofcontents
\section{Introduction}

Consider a random Hermitian $n \times n$ matrix $M_n$ with $n$ real eigenvalues (counting multiplicity)
$$ \lambda_1(M_n) \leq \ldots \leq \lambda_n(M_n).$$ 
By the spectral theorem, one can find an orthonormal basis 
$$ u_1(M_n),\ldots,u_n(M_n) \in S^{2n-1} \subset \C^n$$
of unit eigenvectors of this matrix:
$$ M_n u_i(M_n) = \lambda_i(M_n) u_i(M_n),$$
where $S^{2n-1} := \{ z \in \C^n: |z| = 1 \}$ is the unit sphere of $\C^n$, and we view elements of $\C^n$ as column vectors.  We write $u_{i,p}(M_n) = \langle u_i(M_n), e_p \rangle$ for the $p^{\th}$ coefficient of $u_i(M_n)$ for each $1 \leq i,p \leq n$.  If $M_n$ is not  Hermitian, but is in fact real symmetric, then we can require the $u_i(M_n)$ to have real coefficients, thus taking values in the unit sphere $S^{n-1} := \{ x \in \R^n: |x|=1\}$ of $\R^n$.

Unfortunately, the eigenvectors $u_i(M_n)$ are not unique in either the Hermitian or real symmetric cases; even if one assumes that the spectrum of $M_n$ is \emph{simple}, in the sense that
$$ \lambda_1(M_n) <  \ldots < \lambda_n(M_n),$$ 
one has the freedom to rotate each $u_i(M_n)$ by a unit\footnote{We use $\sqrt{-1}$ to denote the imaginary unit, in order to free up the symbol $i$ as an index variable.} phase $e^{\sqrt{-1}\theta} \in U(1)$.  In the real symmetric case, in which we force the eigenvectors to have real coefficients, one only has the freedom to multiply each $u_i(M_n)$ by a sign $\pm \in O(1)$.  However, one can eliminate this $U(1)$ phase ambiguity or $O(1)$ sign ambiguity (in the case of simple spectrum) by a adopting a variety of viewpoints:
\begin{enumerate}
\item[(i)] One can consider the rank one projection operators
$$ P_i(M_n) := u_i(M_n) u_i(M_n)^*$$
instead of the eigenvectors $u_i(M_n)$, thus the $pq$ coefficient $P_{i,p,q}(M_n)$ of $P_i(M_n)$ is given by the formula
$$ P_{i,p,q}(M_n) = u_{i,p}(M_n) \overline{u_{i,q}(M_n)}.$$
\item[(ii)]  One can consider the equivalence class $[u_i(M_n)] := \{ e^{\sqrt{-1}\theta} u_i(M_n): \theta \in \R \} \in S^{2n-1}/U(1)$ (or $\{ \pm u_i(M_n) \} \in S^{n-1}/O(1)$, in the real symmetric case) of all eigenvectors with eigenvalue $\lambda_i$, instead of considering any one individual eigenvector.
\item[(iii)]  One can perform the \emph{ad hoc} normalization of requiring $u_{i,p}(M_n)$ to be positive real, where $p$ is the first index for which $u_{i,p}(M_n) \neq 0$ (generically we will have $p=1$, and as we will see shortly, for Wigner matrices we will also have $p=1$ with high probability).
\item[(iv)] One can perform the random normalization of replacing $u_i(M_n)$ with a randomly chosen rotation $e^{\sqrt{-1}\theta} u_i(M_n)$ (in the Hermitian case) or $\pm u_i(M_n)$ (in the real symmetric case) (each the random phase or sign being chosen independently of each other, and (if $M_n$ is itself random) of $M_n$).
\end{enumerate}
Note that $P_i$, $[u_i(M_n)]$, the \emph{ad hoc} normalized $u_i(M_n)$, and the randomly normalized $u_i(M_n)$ in viewpoints (i)-(iv) respectively will be uniquely defined as long as the spectrum is simple (indeed, it suffices to have $\lambda_{i-1}(M_n) < \lambda_i(M_n) < \lambda_{i+1}(M_n)$).  In the \emph{proofs} of our main results, we shall adopt viewpoint (i) (which is natural from the perspective of spectral theory).  However, in order to express our results in explicit coordinates, we will adopt the more \emph{ad hoc} viewpoint (iii) or the random viewpoint (iv) in the \emph{statements} of our results.

{\it Notations.}  We consider $n$ as an asymptotic parameter tending
to infinity.  We use $X \ll Y$, $Y \gg X$, $Y = \Omega(X)$, or $X =
O(Y)$ to denote the bound $X \leq CY$ for all sufficiently large $n$
and for some constant  $C$. Notations such as
 $X \ll_k Y, X= O_k(Y)$ mean that the hidden constant $C$ depend on
 another constant $k$. $X=o(Y)$ or $Y= \omega(X)$ means that
 $X/Y \rightarrow 0$ as $n \rightarrow \infty$; the rate of decay here will be allowed to depend on other parameters.

We will need some definitions that capture the intuition that a certain event $E$ occurs very frequently.

\begin{definition}[Frequent events]\label{freq-def}  Let $E$ be an event depending on $n$.
\begin{itemize}
\item $E$ holds \emph{asymptotically almost surely} if $\P(E) = 1-o(1)$.
\item $E$ holds \emph{with high probability} if $\P(E) \geq 1-O(n^{-c})$ for some constant $c>0$.
\item $E$ holds \emph{with overwhelming probability} if $\P(E) \geq 1-O_C(n^{-C})$ for \emph{every} constant $C>0$ (or equivalently, that $\P(E) \geq 1 - \exp(-\omega(\log n))$).
\item $E$ holds \emph{almost surely} if $\P(E)=1$.  
\end{itemize}
\end{definition}

The goal of this paper is to understand the distribution of the eigenvectors $u_i(M_n)$ (as normalized using viewpoint (iii) or (iv), for sake of concreteness) of a class of random matrix ensembles known as \emph{Wigner random matrices}.  

Let us first identify the class of matrices we are working with. 

\begin{definition}[Wigner matrices]\label{def:Wignermatrix}  Let $n \geq 1$ be an integer (which we view as a parameter going off to infinity). An $n \times n$ \emph{Wigner Hermitian matrix} $M_n$ is defined to be a  random Hermitian $n \times n$ matrix $M_n = (\xi_{ij})_{1 \leq i,j \leq n}$, in which the $\xi_{ij}$ for $1 \leq i \leq j \leq n$ are jointly independent with $\xi_{ji} = \overline{\xi_{ij}}$.  For $1 \leq i < j \leq n$, we require that the $\xi_{ij}$ have mean zero and variance one, while for $1 \leq i=j \leq n$ we require that the $\xi_{ij}$ (which are necessarily real) have mean zero and variance $\sigma^2$ for some $\sigma^2>0$ independent of $i,j,n$. (Note that we do not require the $\xi_{ij}$ to be identically distributed, either on or off the diagonal.)

We say that the Wigner matrix ensemble \emph{obeys condition {\condo}} if we have the exponential decay condition
\begin{equation*}
\P(|\xi_{ij}|\ge t^C) \le e^{-t} 
\end{equation*}
for all $1 \leq i,j \leq n$ and $t \ge C'$, and some constants $C, C'$ (independent of $i,j,n$).  We say that the Wigner matrix ensemble \emph{obeys condition {\condone} with constant $C_0$} if one has
$$ \E |\xi_{ij}|^{C_0} \leq C$$
for some constant $C$ (independent of $n$).
\end{definition}

Of course, Condition {\condo} implies Condition {\condone} for any $C_0$, but not conversely.

An important example of a Wigner matrix ensemble is the \emph{Gaussian unitary ensemble} (GUE), which in our notation corresponds to the case when $\xi_{ij}$ has the complex normal distribution $N(0,1)_\C$ with mean zero and variance one for $i>j$, and $\xi_{ij}$ has the real normal distribution $N(0,1)_\R$ with mean zero and variance one for $i=j$.  This ensemble obeys Condition {\condo}, and hence also Condition {\condone} for any $C_0$.  Another important example is the \emph{Gaussian orthogonal ensemble} (GOE), in which $\xi_{ij}$ has the real normal distribution $N(0,1)_\R$ for $i>j$ and the real normal distribution $N(0,2)_\R$ for $i=j$.

The GOE and GUE are continuous ensembles, and it is therefore not difficult to show that the spectrum of such matrices are almost surely simple.
The GUE ensemble is invariant under conjugation by unitary matrices, which implies that each unit eigenvector $u_i(M_n)$ is uniformly distributed on the unit sphere $S^{2n-1}$ of $\C^n$, after quotienting out by the $U(1)$ action $u_i(M_n) \to e^{\sqrt{-1} \theta} u_i(M_n)$ (or equivalently, that $[u_i(M_n)]$ is uniformly distributed using the Haar measure on $S^{2n-1}/U(1)$).  In particular, $P_i(M_n)$ is uniformly distributed on the space of rank one projections in $\C^n$.  Similarly, for GOE, the unit eigenvector $u_i(M_n)$ is uniformly distributed on the unit sphere of $\R^n$ (after quotienting out by the $O(1)$ action $u_i(M_n) \to \pm u_i(M_n)$).  The same argument in fact shows that for GUE, the unitary matrix $(u_1(M_n),\ldots,u_n(M_n))$ is distributed according to Haar measure on the unitary group $U(n)$ (after quotienting out by the left-action of the diagonal unitary matrices $U(1)^n$), and for GOE, the orthogonal matrix $(u_1(M_n),\ldots,u_n(M_n))$ is distributed according to Haar measure on the orthogonal group $O(n)$ (after quotienting out by the left-action of the diagonal sign matrices $O(1)^n$).  In particular, if one uses the random normalization (iv), then $(u_1(M_n),\ldots,u_n(M_n))$ will be distributed according to Haar measure on $U(n)$ (in the GUE case) or $O(n)$ (in the GOE case).

The distribution of coefficients of a matrix distributed using the Haar measure on the unitary or orthogonal groups has been studied by many authors \cite{Borel}, \cite{Stam}, \cite{Gallardo}, \cite{Yor}, \cite{Diaconis}, \cite{DEL}, \cite{Collins}, \cite{Jiang-2006}, \cite{Jiang}.  It is known (see \cite{Borel}) that each coefficient, after multiplication by $\sqrt{n}$, is asymptotically complex normal (in the unitary case) or real normal (in the orthogonal case).  In fact the same is true for the joint distribution of multiple coefficients:

\begin{theorem}\label{sli}\cite{Jiang-2006}, \cite{Jiang}  Let $(u_1,\ldots,u_n)$ be distributed using Haar measure on $O(n)$.  For $1 \leq i,p \leq n$, let $u_{i,p}$ be the $p^{\th}$ coefficient of $u_{i,p}$.  Let $\xi_{i,p} \equiv N(0,1)_\R$ for $1 \leq i,p \leq k$ be iid real gaussian random variables.
\begin{enumerate}
\item[(i)]  If  $k=o(\sqrt{n})$, then $(\sqrt{n} u_{i,p})_{1 \leq i,p \leq k}$ and $(\xi_{i,p})_{1 \leq i,p \leq k}$ differ by $o(1)$ in variation norm.  In other words, if $F: \R^{m^2} \to \R$ is any measurable function with $\|F\|_{L^\infty(\C^{k^2})} \leq 1$, then
\begin{equation}\label{fux}
|\E F\left( (\sqrt{n} u_{i,p})_{1 \leq i,p \leq k} \right) - \E F\left( (\xi_{i,p})_{1 \leq i,p \leq k} \right)| \leq o(1)
\end{equation}
where the error is uniform in $F$.
\item[(ii)] If instead we make the weaker assumption that $k=o(n/\log n)$, then it is possible to couple together $u_{i,p}$ and $\xi_{i,p}$ such that $\sup_{1 \leq i,p \leq k} |\sqrt{n} u_{i,p} - \xi_{i,p}|$ converges to zero in probability. 
\end{enumerate}

Similar conclusions hold for the unitary group $U(n)$, replacing the real gaussian $N(0,1)_\R$ by the complex gaussian $N(0,1)_\C$.
\end{theorem}

In \cite{Jiang-2006} it is also shown that the hypotheses $k=o(\sqrt{n})$, $k=o(n/\log n)$ in the above two results are best possible.  Of course, by symmetry, one can replace the top left $k \times k$ minor $(u_{i,p})_{1 \leq i,p \leq k}$ of the orthogonal matrix $(u_1,\ldots,u_n)$ with any other $k \times k$ minor and obtain the same results.  

As a corollary of Theorem \ref{sli}, we obtain an asymptotic for the distribution of eigenvector coefficients of GUE or GOE (normalizing using viewpoint (iii)).

\begin{corollary}[Eigenvector coefficients of GOE and GUE]\label{goegue} Let $M_n$ be drawn from GOE, and let $(u_1(M_n),\ldots,u_n(M_n))$ be the eigenvectors (normalized using viewpoint (iii); this is almost surely well-defined since the spectrum is almost surely simple).  For $1 \leq i,p \leq n$, let $\xi_{i,p}$ be independent random variables, with $\xi_{i,j} \equiv N(0,1)_\R$ for $p > 1$ and $\xi_{i,p} \equiv |N(0,1)_\R|$ for $p=1$.  Let $1 \leq k \leq n$, and let $1 \leq i_1 < \ldots < i_k \leq n$ and $1 \leq p_1 < \ldots < p_k \leq n$ be indices.
\begin{enumerate}
\item[(i)]  If $k=o(\sqrt{n})$, then $(\sqrt{n} u_{i_a,p_b}(M_n))_{1 \leq a,b \leq k}$ and $(\xi_{i_a,p_b})_{1 \leq a,b \leq k}$ differ by $o(1)$ in variation norm.  
\item[(ii)] If instead we make the weaker assumption that $k=o(n/\log n)$, then it is possible to couple together $M_n$ and $\xi_{i,p}$ such that $\sup_{1 \leq a,b \leq k} |\sqrt{n} u_{i_a,p_b}(M_n) - \xi_{i_a,p_b}|$ converges to zero in probability. 
\end{enumerate}

Similar conclusions hold for GUE, replacing the real gaussian $N(0,1)_\R$ by the complex gaussian $N(0,1)_\C$.  If one uses the random normalization (iv) instead of (iii), then the conclusions are the same, except that the absolute values are not present in the definition of $\xi_{i,1}$.
\end{corollary}

The main objective of this paper is to develop analogues of Corollary \ref{goegue} for the more general Wigner ensembles from Definition \ref{def:Wignermatrix}.  In particular, we would like to consider ensembles $M_n$ which are allowed to be discrete instead of continuous.

One immediate difficulty that arises in the discrete setting is that one can now have a non-zero probability that the spectrum is non-simple.  However, we have the following \emph{gap theorem} from \cite{TVuniv}, \cite{TVedge}, \cite{TVscv}:

\begin{theorem}[Gap theorem]\label{gap}  Suppose that $M_n$ is a Wigner random matrix obeying Condition {\condone} with a sufficiently large constant $C_0$, and let $c_0>0$ be independent of $n$.  Write $A_n := \sqrt{n} M_n$ for the rescaled matrix. Then for any $1 \leq i < n$, one has
$$ \P( |\lambda_{i+1}(A_n) - \lambda_i(A_n)| \leq n^{-c_0} ) \le n^{-c_1}$$
for all sufficiently large $n$, where $c_1>0$ depends only on $c_0$.
\end{theorem}

\begin{proof} In the bulk case $\eps n < i < (1-\eps) n$ assuming Condition {\condo}, see \cite[Theorem 19]{TVuniv}.  For the extension to the edge case, see \cite[Theorem 1.14]{TVedge}.  For the relaxation of Condition {\condo} to Condition {\condone} (with a sufficiently large $C_0$), see \cite[Section 2]{TVscv}. See also \cite{ESY3} for some related level repulsion estimates (assuming some additional regularity hypotheses on $M_n$).
\end{proof}

Of course, one has $\lambda_i(A_n) = \sqrt{n} \lambda_i(M_n)$.
From the above theorem and the union bound, we see that there is an absolute constant $c>0$ such that if $m = O(n^c)$ and $1 \leq i_1 < \ldots < i_m \leq n$, and $M_n$ obeys Condition {\condone} with a sufficiently large $C_0$, then with probability $1-O(n^{-c})$, the eigenvalues $\lambda_{i_1}(M_n),\ldots,\lambda_{i_m}(M_n)$ will all occur with multiplicity one, so that one can meaningfully normalise the eigenvectors $u_{i_1}(M_n),\ldots,u_{i_m}(M_n)$ according to any of the viewpoints (i), (ii), (iii), (iv) mentioned previously, outside of an exceptional event of probability $O(n^{-c})$.  On that exceptional event, we define the orthonormal eigenvector basis $u_i(M_n)$ (and related objects such as the rank one projection $P_i(M_n)$) in some arbitrary (measurable) fashion.

\subsection{Main results}

We now turn to the distribution of the coefficients of the eigenvectors.  As the ensembles $M_n$ may be discrete, the eigenvector coefficients $u_{i,j}(M_n)$ may be discrete also, and so one does not expect to have convergence in variation distance any more.  Instead, we will consider the weaker notion of \emph{vague convergence}, which resembles the condition \eqref{fux}, but with $F$ now required to be compactly supported, and either continuous or smooth.  We need the following definition:

\begin{definition}[Matching moments]  Let $k, l \geq 1$.  Two Wigner random matrices $M_n = (\xi_{ij})_{1 \leq i,j \leq n}$ and $M'_n = (\xi'_{ij})_{1 \leq i,j \leq n}$ are said to \emph{match to order $k$} off the diagonal, and \emph{match to order $l$} on the diagonal, if one has $\E \Re(\xi_{ij})^a \Im(\xi_{ij})^b = \E \Re(\xi'_{ij})^a \Im(\xi'_{ij})^b$ whenever $a, b \geq 0$ and $1 \leq i \leq j \leq n$ are integers such that $a+b \leq k$ (if $i<j$) or $a+b \leq l$ (if $i=j$).
\end{definition}

We can now give our first main result, which partially extends Corollary \ref{goegue} to other Wigner ensembles:

\begin{theorem}[Eigenvector coefficients of Wigner matrices]\label{gg-main} There are absolute constants $\delta, C_0 > 0$ such that the following statement holds.  Let $M_n$ be drawn from a Wigner ensemble obeying Condition {\condone} with constant $C_0$ that matches GOE to order $4$ off the diagonal and to order $2$ on the diagonal (which, in particular, forces $M_n$ to be almost surely real symmetric), and let $(u_1(M_n),\ldots,u_n(M_n))$ be the eigenvectors (normalized using viewpoint (iii)).  For $1 \leq i,p \leq n$, let $\xi_{i,p}$ be independent random variables, with $\xi_{i,p} \equiv N(0,1)_\R$ for $p > 1$ and $\xi_{i,p} \equiv |N(0,1)_\R|$ for $p=1$.  Let $1 \leq k \leq n$, and let $1 \leq i_1 < \ldots < i_k \leq n$ and $1 \leq p_1 < \ldots < p_k \leq n$ be indices.
\begin{enumerate}
\item[(i)] (Vague convergence) If $k=O(n^{\delta})$, then one has
\begin{equation}\label{sil}
\E F( (\sqrt{n} u_{i_a,p_b})_{1 \leq a,b \leq k} ) - \E F( (\xi_{i_a,p_b})_{1 \leq a,b \leq k} )| = o(1)
\end{equation}
whenever $F: \C^{k^2} \to \R$ is a smooth function obeying the bounds
$$ |F(x)| = O(1)$$
and
$$ |\nabla^j F(x)| = O(n^\delta)$$
for all $x \in \C^{k^2}$ and $0 \leq j \leq 5$.
\item[(ii)] If instead we make the stronger assumption that $k =O(1)$, then it is possible to couple together $M_n$ and $\xi_{i,p}$ such that $\sup_{1 \leq a,b \leq k} |\sqrt{n} u_{i_a,p_b}(M_n) - \xi_{i_a,p_b}|$ converges to zero in probability.   In particular, this implies that $(\sqrt{n}u_{i_a,p_b})_{1 \leq a,b \leq k}$ converges to $(\xi_{i_a,p_b})_{1 \leq a,b \leq k}$ in distribution.
\end{enumerate}
The bounds are uniform in the choice of $i_1,\ldots,i_k,p_1,\ldots,p_k$, and in the choice of $F$, provided that the implied constants in the bounds on $F$ are kept fixed.
Similar conclusions hold for GUE, replacing the real gaussian $N(0,1)_\R$ by the complex gaussian $N(0,1)_\C$.

If one uses the random normalization (iv) instead of (iii), then the conclusions are the same, except that the absolute values are not present in the definition of $\xi_{i,1}$.
\end{theorem}

The bound $k=O(1)$ in (ii) can be extended by our method (with some effort) to $k=o(\log^{1/2} n)$, but this still falls far short of the analogous range in Corollary \ref{goegue}. It is reasonable to expect that these bounds are not best possible (particularly if one places some additional regularity hypotheses on the coefficients of $M_n$).

We will deduce Theorem \ref{gg-main} from Corollary \ref{goegue} by establishing a \emph{four moment theorem for eigenvectors}, which is the main technical result of this paper:

\begin{theorem}[Four Moment Theorem for eigenvectors]\label{theorem:main0} There is an absolute constant $C_0>0$ such that for every sufficiently small constant $c_0>0$ there is a constant $\delta >0$ such that the following holds. 
 Let $M_n, M'_n$ be two Wigner random matrices satisfying {\condone} with constant $C_0$ that match to order $4$ off the diagonal and to order $2$ on the diagonal.  Let $k \leq n^\delta$, and let $G: \R^k \times \C^{k} \to \R$ be a smooth function obeying the derivative bounds
\begin{equation}\label{G-deriv}
|\nabla^j G(x)| \leq n^{c_0}
\end{equation}
for all $0 \leq j \leq 5$ and $x \in \R^k \times \C^{k}$.
 Then for any $1 \le i_1, i_2, \ldots, i_k \le n$ and $1 \leq p_1,\ldots,p_k,q_1,\ldots,q_k \leq n$, and for $n$ sufficiently large depending on $\eps,\delta,c_0, C_0$ we have
\begin{equation} \label{eqn:approximation-0}
|\E G( \Phi(\sqrt{n} M_n) ) - \E G( \Phi(\sqrt{n} M'_n) )| \leq n^{-c_0}
\end{equation}
where for any matrix $M$ of size $n$, $\Phi(M) \in \R^k \times \C^{k}$ is the tuple
$$ \Phi(M) := \left( (\lambda_{i_a}(M))_{1 \leq a \leq k}, ( n P_{i_a,p_a,q_a}(M) )_{1 \leq a \leq k} \right).$$
The bounds are uniform in the choice of $i_1,\ldots,i_k,p_1,\ldots,p_k,q_1,\ldots,q_k$.
\end{theorem}

The deduction of Theorem \ref{gg-main} from Corollary \ref{goegue} and Theorem \ref{theorem:main0} is routine and is performed in Section \ref{gg-main-sec}.

Theorem \ref{theorem:main0} is an extension of the \emph{four moment theorem for eigenvalues} established in \cite{TVuniv}, \cite{TVedge}, \cite{TVscv}.  Indeed, the latter theorem is essentially the special case of Theorem \ref{theorem:main0} in which $k=O(1)$, and $G$ only depends on the first $k$ components $(\sqrt{n} \lambda_{i_a}(M_n))_{1 \leq a \leq k}$ of $\Phi(M_n)$.  
Unsurprisingly, the proof of Theorem \ref{theorem:main0} will rely heavily on the machinery developed in \cite{TVuniv}, \cite{TVedge}, \cite{TVscv}. 

 Theorem \ref{theorem:main0} was announced at the AIM workshop ``Random matrices'' in December 2010. 
We have found out that recently  a result in the same spirit has been proved by Knowles and Yin \cite{knowles}
 using  a somewhat different method. Knowles and Yin handled the   $k=O(1)$ case with Condition {\condone} 
  replaced by the stronger Condition {\condo}. Furthermore, they  needeed control on $k+5$ derivatives of $G$ rather than just $5$ derivatives, and also a level repulsion hypothesis similar to \eqref{mo}, \eqref{no} below.  On the other hand, their result holds for generalized Wigner matrices.

  The need for four matching moments in Theorem \ref{theorem:main0} is believed to be necessary; see \cite{TV-nec}.  However, we conjecture that Corollary \ref{goegue} continues to hold without the matching moment hypothesis.

We can use Theorem \ref{theorem:main0} (together with other tools)  to obtain a four moment theorem 
for the resolvent (or Green's function) coefficients
\begin{equation}\label{green-eq}
\left(\frac{1}{\sqrt{n}} M_n-zI\right)^{-1}_{pq} = \sum_{i=1}^n \frac{1}{\lambda_i(A_n)-n z} n P_{i,p,q}(A_n),
\end{equation}
where $A_n$ is the rescaled matrix $A_n := \sqrt{n} M_n$ (so in particular $u_{i}(A_n) = u_i(M_n)$ and $P_i(A_n) = P_i(M_n)$).

More precisely, we have

\begin{theorem}[Four moment theorem for resolvents up to the real axis]\label{green} 
 Let $M_n, M'_n$ be two Wigner random matrices satisfying {\condo} that match to order $4$ off the diagonal and to order $2$ on the diagonal for some sufficiently large $C_0$.  Let $z = E + i \eta$ for some $E \in \R$ and some $\eta > 0$.  We assume the \emph{level repulsion hypothesis} that for any $c>0$, one has with high probability that
\begin{equation}\label{mo}
 \inf_{1 \leq i \leq n} |\lambda_i(\sqrt{n} M_n) - nz| \geq n^{-c}
\end{equation}
and
\begin{equation}\label{no}
 \inf_{1 \leq i \leq n} |\lambda_i(\sqrt{n} M'_n) - nz| \geq n^{-c}.
\end{equation}
Let $1 \leq p,q \leq n$.  Then for any smooth function $G: \C \to \C$ obeying the bounds $\nabla^j G(x) = O(1)$ for all $x \in \C$ and $0 \leq j \leq 5$, one has
$$ \E G\left( (\frac{1}{\sqrt{n}} M_n-zI)^{-1}_{pq} \right) - \E G\left( (\frac{1}{\sqrt{n}} M'_n-zI)^{-1}_{pq} \right) = O(n^{-c_0})$$
for some constant $c_0>0$ independent of $n$.
\end{theorem}

We isolate the $z=0$ case of this theorem as a corollary:

\begin{corollary}[Four moment theorem for the inverse matrix]\label{inverse} 
Under the conditions of Theorem \ref{green}, we have 
$$ \E G\left(\left(\frac{1}{\sqrt{n}} M_n\right)^{-1}_{pq} \right) - \E G\left( \left(\frac{1}{\sqrt{n}} M'_n\right)^{-1}_{pq} \right) = O(n^{-c_0}).$$
\end{corollary} 

We prove Theorem \ref{green} in Section \ref{green-sec}; it is established by combining Theorem \ref{theorem:main0} with an eigenvalue rigidity result from \cite{EYY-rigid} and a local semicircle law from \cite{EYY1}.  This result generalizes a similar four moment theorem from (\cite{EYY1} ( \cite[Theorem 2.3]{EYY1}). Its main strengths as compared against that result are that $\eta$ is allowed to go all the way to zero. In particular, one can take $z$ to be zero, thus giving control of the coefficients of the inverse matrix $M_n^{-1}$. 
On the other hand, the result in \cite[Theorem 2.3]{EYY1} does not require the hypothesis \eqref{mo}, \eqref{no}, and allows the entries in $M_n$ (or $M'_n$) to have different variances, and the bounds are slightly sharper. 

\begin{remark} The same proof allows one to control the joint distribution of $k$ coefficients of several resolvents with $k = O(n^\delta)$ for some sufficiently small $\delta > 0$, assuming a level repulsion estimate at each energy $z$. \end{remark}

 The hypotheses \eqref{mo}, \eqref{no} are natural, as when these claims fail one would expect the resolvent to be unusually large.  However, in practice such hypotheses are in fact automatic and can thus be omitted.  For instance, we have

\begin{proposition}[Sufficient conditions for level repulsion]  Let $M_n$ be a Wigner matrix whose entries $\xi_{ij}$ are identically distributed in the off-diagonal $1 \leq i < j \leq n$, and also identically distributed on the diagonal $1 \leq i=j \leq n$.  Let $z \in \C$.
\begin{enumerate}
 \item[(i)] If the real and imaginary parts of the off-diagonal coefficients of $M_n$ are iid and supported on at least three points, and $M_n$ obeys Condition \condone for a sufficiently large $C_0$, then \eqref{mo} holds.
 \item[(ii)] If $z=0$ and $M_n$ obeys Condition \condone for $C_0=4$, then \eqref{mo} holds.
\end{enumerate}
\end{proposition}

\begin{proof}   For (i), see \cite{mehta}; an earlier result assuming smoothness and decay on the coefficients (but without the requirement of iid real and imaginary parts) was established in \cite{ESY3} (see also \cite{maltsev} for a more refined result).  The claim (ii) was recently established (by a rather different method) in \cite{vershynin} 
 \end{proof} 

It is in fact likely that \eqref{mo} and \eqref{no} in fact always hold whenever Condition \condone is satisfied for a sufficiently large $C_0$, but we do not attempt to establish this fact here.

As another application, we can  use our new results to obtain central limit theorems concerning eigenvectors.  Here is a sample result. 

\begin{theorem}  \label{theorem:CLT2} Let $M_n$ be a random symmetric matrix obeying hypothesis {\condone} for a sufficiently large constant $C_0$, 
 which matches the Gaussian Orthogonal Ensemble(GOE) to fourth order.  Assume furthermore that the atom distributions of $M_n$ are symmetric (i.e. $\xi_{ij} \equiv -\xi_{ij}$ for all $1 \leq i,j \leq n$), and identically distributed (or more precisely, that the $\xi_{ij}$ for $i>j$ are identically distributed, and the $\xi_{ij}$ for $i=j$ are also identically distributed).  Let $i = i_n$ be an index (or more precisely, a sequence of indices) between $1$ and $n$, and let $a= a_n \in S^{n-1}$ be a unit vector in $\R^n$ (or more precisely, a sequence of unit vectors).  Assume either that
\begin{itemize}
\item[(a)] $u_i(M_n)$ is normalized using the procedure (iii), and $a \cdot e_1 = o(1)$; or
\item[(b)] $u_i(M_n)$ is normalized using the procedure (iv).
\end{itemize}
Then $\sqrt n u_{i}(M_n) \cdot a$ tends to $N(0,1)_\R$ in distribution as $n \to \infty$. 
\end{theorem} 

Note that if one normalizes $u_i(M_n)$ using procedure (iii), then $u_i(M_n) \cdot e_1 = u_{i,1}(M_n)$ will be non-negative (and thus cannot converge in distribution to $N(0,1)_\R$), which helps explain the presence of the condition $a \cdot e_1$ in condition (a).

As an example to illustrate Theorem \ref{theorem:CLT2}, we can take $a = a_n := \frac{1}{\sqrt n}(1,\ldots, 1) \in S^{n-1}$, and $i := \lfloor n/2\rfloor$.  Then Theorem \ref{theorem:CLT2} asserts that 
the sum of the entries of the middle eigenvector $u_{\lfloor n/2\rfloor}(M_n)$ (using either normalization (iii) or normalization (iv)) is gaussian in the limit.

We prove Theorem \ref{theorem:CLT2} in Section \ref{clt2-sec} as a consequence of Corollary \ref{goegue} and a general central limit theorem on averages of approximately independent symmetric random variables (Proposition \ref{propo1}) which may be of independent interest.  It should be possible to extend this result to more general ensembles $M_n$ (and to obtain the analogous results for ensembles that match GUE rather than GOE), and to obtain central limit theorems for the joint distribution of several statistics of the form $\sqrt{n} u_i(M_n) \cdot a_n$, but we do not pursue this matter here.

\begin{remark} The four moment theorem can be extended to handle the singular values of iid covariance matrices, instead of the eigenvalues of Wigner matrices; see \cite{TVscv}.  It is possible to use that extension to establish an analogue of Theorem \ref{theorem:main0} for the singular values and singular vectors of such matrices, and an analogue of Theorem \ref{green} for the inverses of such matrices (which, as is well known, can be expressed in terms of the singular value decomposition of the matrix).  We omit the details.
\end{remark}

\section{Proof of Theorem \ref{gg-main}}\label{gg-main-sec}

 We establish the claim in the GOE case only, as the GUE case is  similar.  We shall also establish the claim just for the normalization (iii), as the normalization (iv) can be treated similarly (or deduced directly from the results for (iii)).

Let $C_0$ be as in Theorem \ref{theorem:main0}, let $\delta > 0$ be sufficiently small, and let $M_n$, $u_1(M_n),\ldots,u_n(M_n)$, $\xi_{i,j}$, $k$, $i_1,\ldots,i_k,p_1,\ldots,p_k$, and $F$ be as in that theorem.  Let $M'_n$ be drawn from GOE, and write $A_n := \sqrt{n} M_n$ and $A'_n := \sqrt{n} M'_n$.  We initially assume that $k = O(n^\delta)$.  By adding a dummy index and relabeling if necessary, we may assume without loss of generality that $p_1=1$.

The normalization given by viewpoint (iii) is unstable when $u_{i,1}$ is close to zero, and so we will need to eliminate this possibility first.  To this end, let $F_0: \C^k \to \R$ be a function of the form
$$ F_0(x_1,\ldots,x_k) := \prod_{a=1}^k (1 - \chi(n^{10\delta} x_a))$$
where $\chi: \C \to [0,1]$ is a smooth cutoff which equals one on the ball $B(0,1)$ but vanishes outside of $B(0,2)$.  Observe from the union bound that 
$$ \E F_0\left( ( |\xi_{i_a,1}|^2 )_{1 \leq a \leq k} \right) = 1 - o(1) $$

and thus by Corollary \ref{goegue}, we have an analogous estimate for the quantities $n P_{i_a,1,1}(A'_n) = (\sqrt{n} u_{i_a,1}(A'_n))^2$:
$$ \E F_0\left( ( n P_{i_a,1,1}(A'_n) )_{1 \leq a \leq k} \right) = 1 - o(1).$$

Applying Theorem \ref{theorem:main0} (assuming that $\delta$ is sufficiently small depending on $c_0$, so that the losses of $n^{O(\delta)}$ coming from bounding the derivatives of $F_0$ can be absorbed into the $n^{-c_0}$ factor) we conclude that
\begin{equation}\label{fon}
 \E F_0\left( ( n P_{i_a,1,1}(A_n) )_{1 \leq a \leq k} \right) = 1 - o(1),
\end{equation}

and thus we have
$$ \inf_{1 \leq a \leq k} |n P_{i_a,1,1}(A_n)| \geq n^{-10\delta} $$
or equivalently that
\begin{equation}\label{n5}
 \inf_{1 \leq a \leq k} \sqrt{n} u_{i_a,1}(A_n) \geq n^{-5\delta}
 \end{equation}
asymptotically almost surely.

Now we can prove part (i) of Theorem \ref{gg-main}.  From \eqref{fon} one has
$$ \E \left|1 - F_0\left( ( n P_{i_a,1,1}(A_n) \right)_{1 \leq a \leq k} )\right| = o(1)$$
and thus (by the boundedness of $F$)
$$ \E \left(1 - F_0\left( ( n P_{i_a,1,1}(A_n) )_{1 \leq a \leq k} \right)\right) F\left( (\sqrt{n} u_{i_a,p_b}(A_n))_{1 \leq a,b \leq k} \right) = o(1).$$

A similar argument gives
$$ \E \left(1 - F_0\left( ( |\xi_{i_a,1}|^2 )_{1 \leq a \leq k} \right)\right) F\left( ( \xi_{i_a,p_b} )_{1 \leq a,b \leq m} \right) = o(1).$$
Thus, to prove \eqref{sil}, it suffices by the triangle inequality to show that
$$
\E \tilde F\left( (\sqrt{n} u_{i_a,p_b}(A_n))_{1 \leq a,b \leq k} \right) - \E \tilde F\left( (\xi_{i_a,p_b})_{1 \leq a,b \leq k} \right) = o(1)$$
where $\tilde F: \C^{k^2} \to \C$ is the function
$$ \tilde F\left( (x_{a,b})_{1 \leq a,b \leq k} \right) =  F_0\left( (|x_{a,1}|^2)_{1 \leq a \leq k} \right) F\left( (x_{a,b})_{1 \leq a,b \leq k} \right).$$

From Corollary \ref{goegue} we have
$$
\E \tilde F\left( (\sqrt{n} u_{i_a,p_b}(A'_n))_{1 \leq a,b \leq k} \right) - \E \tilde F\left( (\xi_{i_a,p_b})_{1 \leq a,b \leq k} \right) = o(1)$$
so it suffices to show that
$$
\E \tilde F\left( (\sqrt{n} u_{i_a,p_b}(A'_n))_{1 \leq a,b \leq k} \right) - \tilde F\left( (\sqrt{n} u_{i_a,p_b}(A_n))_{1 \leq a,b \leq k}  \right) = o(1).$$
But observe that in the support of $\tilde F$, $\sqrt{n} u_{i_a,1}$ is non-zero (and indeed, we have \eqref{n5}), and the components $\sqrt{n} u_{i_a,j_b}$ can then be recovered from the projection coefficients $n P_{i_a,p_b,1} = n u_{i_a,p_b} u_{i_a,1}$ by the formula
$$ \sqrt{n} u_{i_a,p_b} = \frac{n P_{i_a,p_b,1}}{\sqrt{n P_{i_a,1,1}}}.$$
Thus we can write
$$ \tilde F\left( (\sqrt{n} u_{i_a,p_b}(A'_n))_{1 \leq a,b \leq k} \right) = G\left( (n P_{i_a,p_b,1}(A'_n))_{1 \leq a,b \leq k} \right)$$
where
$$ G\left( (x_{a,b})_{1 \leq a,b \leq k} \right) := \tilde F\left( \left( \frac{x_{a,b}}{\sqrt{x_{a,1}}} \right)_{1 \leq a,b \leq k} \right)$$
with the understanding that the right-hand side vanishes when any of the $x_{a,1}$ vanish.  Similarly for $A_n$. From construction we can easily verify that
$$ |\nabla^j G(x)| = O(n^{O(\delta)} )$$
for all $0 \leq j \leq 5$.  The claim (i) now follows from Theorem \ref{theorem:main0} (assuming $\delta$ sufficiently small depending on $c_0$).

Finally, we prove Claim (ii) of Theorem \ref{gg-main}.  Assume that $k = O(1)$.  Let $\eps = \eps(n) = o(1) > 0$ be a slowly decaying function of $n$ to be chosen later.  Observe that asymptotically almost surely, one has
$$ \sum_{a=1}^k \sum_{b=1}^k |\xi_{i_a,p_b}|^2 = O( 1/\eps^2 ),$$
and thus by (i), we conclude that asymptotically almost surely, we also have
$$ \sum_{a=1}^k \sum_{b=1}^k |\sqrt{n} u_{i_a,p_b}|^2 = O( 1/\eps^2 ),$$
thus the vector $\vec u := (\sqrt{n} u_{i_a,p_b})_{1 \leq a,b \leq k}$ asymptotically almost surely lies in a ball $B$ of radius $O(1/\eps)$ in $\C^{k^2}$.  Since $k=O(1)$, we may thus cover this ball by $O(\eps^{-O(1)})$ disjoint half-open cubes $Q_1,\ldots,Q_m$ of sidelength $\eps$; thus
$$ \P\left( \vec u \in \bigcup_{i=1}^m Q_i \right) = 1-o(1).$$
Write $\vec \xi := (\xi_{i_a,p_b})_{1 \leq a,b \leq k}$.  Applying Claim (i) (approximating the indicator functions $1_{Q_i}$ from above and below by smooth functions) and using the continuous nature of $\vec \xi$, we see that
$$ \sum_{i=1}^m |\P( \vec u \in Q_i ) - \P( \vec \xi \in Q_i )| = o(1)$$
if $\eps$ is sufficiently slowly decaying in $n$.  From this, we see that we may couple $\vec u$ and $\vec \xi$ together in such a fashion that with probability $1-o(1)$, $\vec u$ and $\vec \xi$ lie in the same cube $Q_i$, which in particular implies that $\vec u - \vec \xi = o(1)$.  The claim follows.

\begin{remark}  By using the polynomial decay rate $O(n^{-c})$ in the bounds appearing in the proof of Claim (i), it is possible to relax the condition $k=O(1)$ in Claim (ii) to $k = o(\log^{1/2} n)$, as this allows us to absorb factors of the form $O(\exp(O(k^2)))$ arising from the high dimension of the domain $\C^{k^2}$ into the $O(n^{-c})$ factors.  We omit the details.
\end{remark}

\section{Proof of four moment theorem}

In this section we prove Theorem \ref{theorem:main0}.  We will follow closely the arguments from \cite{TVuniv}, \cite{TVedge}, \cite{TVscv}.

\subsection{Heuristic discussion}

Let us first review the general strategy from \cite{TVuniv} for handling the eigenvalues.   As in \cite{TVuniv}, we introduce the normalised matrices
$$ A_n := \sqrt{n} M_n, A'_n =\sqrt n M'_n $$
(whose mean eigenvalue spacing is comparable to $1$).
  For sake of exposition let us restrict attention to the case $k=1$, thus we wish to show that the expectation $\E G(\lambda_i(A_n))$ of the random variable $G(\lambda_i(A_n))$ only changes by $O( n^{-c_0})$ if one replaces $A_n$ with another random matrix $A'_n$ with moments matching up to fourth order off the diagonal (and up to second order on the diagonal).  To further simplify the exposition, let us suppose that the coefficients $\zeta_{pq}$ of $A_n$ (or $A'_n$) are real-valued rather than complex-valued.

At present, $A'_n$ differs from $A_n$ in all $n^2$ components.  But suppose we make a much milder change to $A_n$, namely replacing a single entry $\sqrt{n} \zeta_{pq}$ of $A_n$ with its counterpart $\sqrt{n} \zeta'_{pq}$ for some $1 \leq p \leq q \leq n$.  If $p \neq q$, one also needs to  replace the companion entry $\sqrt{n} \zeta_{qp} = \sqrt{n} \overline{\zeta}_{pq}$ with $\sqrt{n} \zeta'_{qp} = \sqrt{n} \overline{\zeta}'_{pq}$, to maintain the Hermitian property.  This creates another random matrix $\tilde A_n$ which differs from $A_n$ in at most two entries.  Note that $\tilde A_n$ continues to obey Condition \condone, and has matching moments with either $A_n$ or $A'_n$ up to fourth order off the diagonal, and up to second order on the diagonal.

Suppose that one could show that $\E G( \lambda_i(A_n) )$ differed from $\E G( \lambda_i(\tilde A_n) )$ by at most $n^{-2-c_0}$ when $p \neq q$ and by at most $n^{-1-c_0}$ when $p=q$.  Then, by applying this swapping procedure once for each pair $1 \leq p \leq q \leq n$ and using the triangle inequality, one would obtain the desired bound $|\E G(\lambda_i(A_n) ) - \E G( \lambda_i(A'_n) ) |= O( n^{-c_0} )$.

Now let us see why we would expect $\E G( \lambda_i(A_n) )$ to differ from $\E G( \lambda_i(\tilde A_n) )$ by such a small amount.  For sake of concreteness let us restrict attention to the off-diagonal case $p \neq q$, where we have four matching moments; the diagonal case $p=q$ is similar but one only assumes two matching moments, which is ultimately responsible for the $n^{-1-c_0}$ error rather than $n^{-2-c_0}$.

Let us freeze (or condition on) all the entries of $A_n$ except for the $pq$ and $qp$ entries.  For any complex number $z$, let $A(z)$ denote the matrix which equals $A_n$ except at the $pq$, $qp$, entries, where it equals $z$ and $\overline{z}$ respectively.  (Actually, with our hypotheses, we only need to consider real-valued $z$.)  Thus it would suffice to show that
\begin{equation}\label{eff}
\E F( \sqrt{n} \zeta_{pq} ) = \E F( \sqrt{n} \zeta'_{pq} ) + O( n^{-2-c_0} )
\end{equation}
for all (or at least most) choices of the frozen entries of $A_n$, where $F(z) := G( \lambda_i( A(z) ) )$.  (A standard argument allows us to restrict attention to values of $z$ of size $O( n^{1/2+O(1/C_0)} )$.)

Suppose we could show the derivative estimates
\begin{equation}\label{fzk}
\frac{d^l}{dz^l} F(z) = O( n^{-l+O(c_0)+O(1/C_0)+o(1)} )
\end{equation}
for $l=1,2,3,4,5$.   Then by Taylor's theorem with remainder, we would have
$$ F(z) = F(0) + F'(0) z + \ldots + \frac{1}{4!} F^{(4)}(0) z^4 + O( n^{-5+O(c_0)+O(1/C_0)+o(1)} |z|^5 )$$
and so in particular (using the hypothesis $z = O(n^{1/2+O(1/C_0)})$)
$$ F( \sqrt{n} \zeta_{pq} ) = F(0) + F'(0) \sqrt{n} \zeta_{pq} + \ldots + \frac{1}{4!} F^{(4)}(0) \sqrt{n}^4 \zeta_{pq}^4 + O( n^{-5/2+O(c_0)+O(1/C_0)+o(1)} )$$
and similarly for $F( \sqrt{n} \zeta'_{pq} )$.  Since $n^{-5/2+O(c_0)+O(1/C_0)+o(1)} = O( n^{-2-c_0} )$ for $C_0$ and $n$ large enough and $c_0$ small enough, we thus obtain the claim \eqref{eff} thanks to the hypothesis that the first four moments of $\zeta_{pq}$ and $\zeta'_{pq}$ match.  (Note how this argument barely fails if only three moments are assumed to match, though it is possible that some refinement of this argument might still succeed by exploiting further cancellations in the fourth order term $\frac{1}{4!} F^{(4)}(0) \sqrt{n}^4 \zeta_{pq}^4$.)

We can use exactly this strategy for eigenvectors, replacing $\lambda_i$ by $P_{i,p,q}$ (say). A key point here is that the derivatives of $F$ is computed based on the basic relation 
$A_n u_i = \lambda_i u_i $ and the chain rule. Thus, in order to bound the derivatives of $F$ (for the four moment theorem for eigenvalues), we needed to 
obtain estimates for the derivatives of both $\lambda_i$ and $u_i$. The same estimates can be used to bound the derivatives of $F$ in the proof of the four moment theorem for 
eigenvectors.  In order to make $k$ as large as $n^{\Omega (1)}$, one needs to follow the proof closely and make 
certain adjustments. 

\subsection{Formal proof}

We now turn to the details. 
The first step is to truncate away the event that an eigenvalue gap is unexpectedly small. Define

$$ Q_i(A_n) := \sum_{j \neq i} \frac{1}{|\lambda_j(A_n) - \lambda_i(A_n)|^2}.$$
This quantity is usually bounded from above:

\begin{lemma}[Bound on $Q$]\label{q-bound}  If $c_0>0$, then one has
$$ Q_{i_a}(A_n) \leq n^{c_0} $$
with high probability for all $1 \leq a \leq k$.
\end{lemma}

\begin{proof}  See \cite[Lemma 49]{TVuniv}.  Strictly speaking, this lemma assumed Condition {\condo} and was restricted to the bulk of the spectrum, but this was solely because at the time of writing of that paper, the gap theorem (Theorem \ref{gap}) was only established in that setting.  Inserting Theorem \ref{gap} as a replacement for \cite[Theorem 19]{TVuniv} in the proof of \cite[Lemma 49]{TVuniv}, we obtain the claim.
\end{proof}

In view of this lemma (and its obvious counterpart for $M'_n$), we can (if $k = O(n^\delta)$ for a sufficiently small $\delta$) deduce the four moment theorem from the following truncated version:

\begin{theorem}[Truncated Four Moment Theorem for eigenvectors]\label{theorem:main1} There is an absolute constant $C_0>0$ such that for every sufficiently small constant $c_1>0$ there is a constant $\delta >0$ such that the following holds. 
 Let $M_n, M'_n$ be two Wigner random matrices satisfying {\condone} with constant $C_0$ that match to order $4$ off the diagonal and to order $2$ on the diagonal.  Let $k \leq n^{\delta}$, and let $G: \R^k \times \R^k \times \C^{k} \to \R$ be a smooth function obeying the derivative bounds
\begin{equation}\label{G-deriv-trunc}
|\nabla^j G(x)| \leq n^{c_1}
\end{equation}
for all $0 \leq j \leq 5$ and $x \in \R^k \times \R^k \times \C^{k}$, and which is supported on the region
$$ \left\{ \left( (x_a)_{1 \leq a \leq k}, (y_a)_{1 \leq a \leq k}, (z_{ab})_{1 \leq a,b \leq k} \right) : |y_a| \leq n^{c_1} \hbox{ for all } 1 \leq a \leq k \right\}.$$
 Then for any $1 \le i_1, i_2, \ldots, i_k \le n$ and $1 \leq p_1,\ldots,p_k,q_1,\ldots,q_k \leq n$, and for $n$ sufficiently large depending on $\eps,\delta,c_0$ (and the constant $C$ in Definition \ref{def:Wignermatrix}) we have
\begin{equation} \label{eqn:approximation}
|\E G( \tilde \Phi(A_n) ) - \E G( \tilde \Phi(A'_n) )| \leq n^{-c_1}
\end{equation}
where for any matrix $A_n$, $\tilde \Phi(A_n) \in \R^k \times \C^{k}$ is the tuple
$$ \tilde \Phi(A_n) := \left( (\lambda_{i_a}(A_n))_{1 \leq a \leq k}, (Q_{i_a}(A_n))_{1 \leq a \leq k}
( n P_{i_a}(A_n)_{p_a,q_a} )_{1 \leq a \leq k} \right).$$
\end{theorem}

The reduction of Theorem \ref{theorem:main0} to Theorem \ref{theorem:main1} using Lemma \ref{q-bound} proceeds exactly as in \cite[\S 3.3]{TVuniv} and is omitted\footnote{Note that now that $k$ is as large as $O(n^\delta)$, some factors of $O(n^{O(\delta)})$ may be lost in the bounds, but this can be absorbed by the $n^{-c_1}$ gain in the conclusion of Theorem \ref{theorem:main1}.}.

As in \cite{TVuniv}, we adopt the Lindeberg strategy of swapping each matrix entry of $M_n$ (and its transpose) one at a time.  A key definition is that of a \emph{good configuration}, which is a slightly modified version of the same concept from \cite{TVuniv}.  Let $\eps_1, C_1$ be parameters (independent of $n$) to be selected later.  For a given $n$, we fix $k, i_1,\ldots,i_k, G$ as in Theorem \ref{theorem:main1}.

\begin{definition}[Good configuration]  For a complex parameter $z$, let $A(z)$ be a (deterministic) family of $n \times n$ Hermitian matrices of the form
$$ A(z) = A(0) + z e_p e_q^* + \overline{z} e_q e_p^*$$
where $e_p, e_q$ are unit vectors.  We say that $A(z)$ is a \emph{good configuration} if for every $1 \leq a \leq k$ and every $|z| \leq n^{1/2+\eps_1}$ whose real and imaginary parts are multiples of $n^{-C_1}$, we have the following properties:
\begin{itemize}
\item (Eigenvalue separation)  For any $1 \leq i \leq n$ with $|i-i_a| \geq n^{\eps_1}$, we have
\begin{equation}\label{noon}
 |\lambda_i(A(z)) - \lambda_{i_a}(A(z))| \geq n^{-\eps_1} |i-i_a|.
\end{equation}
\item (Delocalization)  
There exists an orthonormal eigenfunction basis $u_1(A(z)),\ldots,u_n(A(z))$ such that
\begin{equation}\label{uimenj}
\sup_{1 \leq i,j \leq n} |u_{i,j}(A(z))| \ll n^{-1/2+\eps_1}
\end{equation}
\end{itemize}
\end{definition}

To show Theorem \ref{theorem:main1}, it then suffices to establish the following two claims.  The first claim, which we shall establish shortly, is an analogue of \cite[Proposition 46]{TVuniv}, and involves a single (deterministic) good configuration $A(z)$:

\begin{proposition}[Replacement given a good configuration]\label{good-replace}  
  Suppose that $C_1$ is sufficiently large, and let $\eps_1 > 0$.  Let $A(z)$ be a good configuration.  Then one has
$$|\E G( \tilde \Phi(A(\zeta)) ) - \E G( \tilde \Phi(A(\zeta') )| = O( n^{-(r+1)/2 + O(\eps_1) + O(\delta)} )$$
whenever $\zeta, \zeta'$ are random complex variables that match to order $r$ for some $r=2,3,4$, and bounded almost surely by $O(n^{1/2+\eps_1})$.
\end{proposition}

The second claim is an analogue of \cite[Proposition 48]{TVuniv}:

\begin{proposition}[Good configurations occur very frequently]\label{goodconf}
Let $\eps_1 > 0$ and $C_1 \geq 1$, and assume that $C_0 \geq 1$ is sufficiently large depending on $\eps_1$.  Let $A(0) = (\zeta_{ij})_{1 \leq i,j \leq n}$ be a random Hermitian matrix with independent upper-triangular entries and $|\zeta_{ij}| \leq n^{1/2+10/C_0}$ for all $1 \leq i,j \leq n$, with $\zeta_{pq}=\zeta_{qp}=0$, but with $\zeta_{ij}$ having mean zero and variance $n$ for all other $ij$, and also being distributed continuously in the complex plane.  Then $A(0)$ is a good configuration with overwhelming probability.
\end{proposition}

\begin{proof} This is almost identical to the arguments\footnote{The hypotheses in \cite[\S 5]{TVuniv} did not have the loss of $n^{10/C_0}$ in the bound for $\zeta_{ij}$, but such bounds only cause losses of $O(n^{O(1/C_0)})$ in the final bounds, which can be absorbed into the $n^{\eps_1}$ factors in the definition of a good configuration if $C_0$ is sufficiently large depending on $\eps_1$.} in \cite[\S 5]{TVuniv}.  Those arguments already give \eqref{noon} with overwhelming probability.  To obtain the delocalization of eigenvectors, one can use \cite[Proposition 1.12]{TVedge} (see also the proof of \cite[Corollary 63]{TVuniv} to deal with the fact that the $pq$ and $qp$ entries are not random).    
\end{proof}

With these two propositions, we can establish Theorem \ref{theorem:main1} by first using Condition {\condone} to truncate to the case when $M_n, M'_n$ have entries of size $O(n^{10/C_0})$ (say), and perturbing them to be continuous, and then replacing the entries of $M_n$ with $M'_n$ one at a time just as in \cite[\S 3.3]{TVuniv}.  Because the entries of $M_n$ and $M'_n$ match to order $4$ off the diagonal and to order $2$ on the diagonal, each of the $O(n^2)$ off-diagonal replacements costs an error of $O(n^{-5/2 + O(\eps_1) + O(\delta)})$, while each of the $O(n)$ diagonal replacements costs an error of $O(n^{-3/2+O(\eps_1)+O(\delta)})$, and so the net error is acceptable if $\eps_1,\delta$ are sufficiently small (and if $C_0$ is large enough that Proposition \ref{goodconf} applies).

It remains to establish Proposition \ref{good-replace}.  As in \cite{TVuniv}, we will need derivative bounds on various spectral statistics of $A(z)$.
For each $1 \leq i \leq n$, we introduce the resolvent-type matrices
$$ R_i(A) = \sum_{j \neq i} \frac{1}{\lambda_j(A)-\lambda_i(A)} P_j(A).$$

\begin{lemma} Let $1 \leq i \leq n$, and let $A_0$ be a Hermitian matrix
which has a simple eigenvalue at $\lambda_i(A_0)$.
Then $\lambda_i(A)$, $P_i(A)$, $R_i(A)$, and $Q_i(A)$ depend smoothly on $A$
in a neighborhood of $A_0$.
\end{lemma}

\begin{proof}  See \cite[Lemma 54]{TVuniv}.
\end{proof}

Now we turn to more quantitative bounds on derivatives.  If $f(z)$ is a scalar, vector, or matrix-valued function depending smoothly (but not holomorphically) on a complex parameter $z$, we define the derivatives $\nabla^m f$ of $f$ to be the vector (or tensor)-valued quantity
$$\nabla^m f(z) := \left(\frac{\partial^m f}{\partial \Re(z)^l \partial \Im(z)^{m-l}}\right)_{l=0}^m.$$
We have a crude bound:

\begin{lemma}[Crude bound]\label{crude}  Let $A = A(z)$ be an $n \times n$ Hermitian matrix varying (real)-linearly in $z$ (thus $\nabla^k A = 0$ for $k \geq 2$), with
$$ \| \nabla A \|_{op} \leq V$$
for some $V > 0$.  Let $1 \leq i \leq n$.
At some fixed value of $z$, suppose we have the spectral gap condition
\begin{equation}\label{laj}
|\lambda_j(A(z)) - \lambda_i(A(z))| \geq r
\end{equation}
for all $j \neq i$ and some $r > 0$ (in particular, $\lambda_i(A(z))$ is a simple eigenvalue).
Then for all $k \geq 1$ we have (at this fixed choice of $z$)
\begin{equation}\label{li-crude}
 |\nabla^{k} \lambda_i(A(z))| \ll_k V^k r^{1-k}
\end{equation}
and
\begin{equation}\label{pi-crude}
 \| \nabla^{k} P_i(A(z)) \|_{op} \ll_k V^k r^{-k}
\end{equation}
and
\begin{equation}\label{ri-crude}
 \| \nabla^{k} R_i(A(z)) \|_{op} \ll_k V^k r^{-k-1}.
\end{equation}
and
\begin{equation}\label{xi-crude}
 |\nabla^{k} Q_i| \ll_k n V^k r^{-k-2}.
\end{equation}
\end{lemma}

\begin{proof} See \cite[Lemma 56]{TVuniv}.
\end{proof}

In practice, this crude bound is insufficient, and we will need the following more advanced bound:

\begin{proposition}[Better bound]\label{better-cor} Let $A = A(z)$ be an $n \times n$ matrix depending on a complex parameter $z$ of the form
$$ A(z) = A(0) + z e_p e_q^* + \overline{z} e_q e_p^*$$
for some vectors $e_p,e_q$.  We abbreviaate $\lambda_i = \lambda_i(A(z))$, $P_i = P_i(A(z))$, etc.

Let $1 \leq i \leq n$.  At some fixed value of $z$, suppose that $\lambda_i = \lambda_i(A(z))$ is a simple eigenvalue, and that we have a partition
$$ I = P_i + \sum_{\alpha \in J} P_\alpha$$
where $J$ is a finite index set, and $P_\alpha$ are orthogonal projections to invariant spaces on $A$ (i.e. to spans of eigenvectors not corresponding to $\lambda_i$).  Suppose that on the range of each $P_\alpha$, the eigenvalues of $A-\lambda_i$ have magnitude at least $r_\alpha$ for some $r_\alpha > 0$.  Suppose also that we have the incompressibility bounds
$$ \| P_\alpha e_p \|, \| P_\alpha e_q \| \leq w d_\alpha^{1/2} $$
for all $\alpha \in J \cup \{i\}$ and some $w > 0$ and $d_\alpha \geq 1$, with $d_i := 1$.  Then at this value of $z$, and for all $k \geq 1$, we have the bounds
\begin{align}
 |\nabla^{k} \lambda_i| &\ll_k \left(\sum_{\alpha \in J} \frac{d_\alpha}{r_\alpha}\right)^{k-1} w^{2k}\label{luck2}\\
\| P_i(\nabla^k P_i) P_i \|_F &\ll_k \left(\sum_{\alpha \in J} \frac{d_\alpha}{r_\alpha}\right)^{k} w^{2k}\label{extra-1}\\
\| P_\alpha(\nabla^k P_i) P_i \|_F =
\| P_i(\nabla^k P_i) P_\alpha \|_F
 &\ll_k \frac{d_\alpha^{1/2}}{r_\alpha} \left(\sum_{\alpha \in J} \frac{d_\alpha}{r_\alpha}\right)^{k-1} w^{2k}\label{extra-2}\\
\| P_\alpha(\nabla^k P_i) P_\beta \|_F
 &\ll_k \frac{d_\alpha^{1/2} d_\beta^{1/2}}{r_\alpha r_\beta} \left(\sum_{\alpha \in J} \frac{d_\alpha}{r_\alpha}\right)^{k-2} w^{2k}\label{extra-3}\\
|\nabla^k Q_i| &\ll_k \left(\sum_{\alpha \in J} \frac{d_\alpha}{r_\alpha}\right)^{k+2} w^{2k}\label{furioso}
\end{align}
for all $k \geq 0$ and all $\alpha,\beta \in J$ at this value of $z$.  Here $\|T\|_F := (\tr TT^*)^{1/2}$ is the Frobenius norm of $T$.
\end{proposition}

\begin{proof} See \cite[Corollary 58]{TVuniv}, which is a special case of \cite[Lemma 57]{TVuniv}.  The bounds \eqref{extra-1}, \eqref{extra-2}, \eqref{extra-3} do not appear explicitly in \cite[Corollary 58]{TVuniv}, but come directly from the equations \cite[(73), (74), (75)]{TVuniv} in \cite[Lemma 57]{TVuniv} after plugging in the parameters indicated in the proof of \cite[Corollary 58]{TVuniv}.
\end{proof}

We can now prove Proposition \ref{good-replace} and thus Theorem \ref{theorem:main0}.  This will be a repetition of the material in \cite[Section 4.3]{TVuniv}; we sketch the main points here.  

Fix $k \geq 1$, $r=2,3,4$ and $\eps_1 > 0$, and suppose that $C_1$ is sufficiently large.  We assume $A(0), e_p, e_q, i_1,\ldots,i_k, G, F, \zeta, \zeta'$ are as in the proposition.

We may of course assume that $F(z_0) \neq 0$ for at least one $z_0$ with $|z_0| \leq n^{1/2 +\eps_1}$, since the claim is vacuous otherwise.

Using Taylor expansion and the chain rule exactly as in \cite[Section 4.3]{TVuniv}, it suffices to show that

\begin{lemma}  Suppose that $F(z_0) \neq 0$ for at least one $z_0$ with $|z_0| \leq n^{1/2 +\eps_1}$.  Then for \emph{all} $z$ with $|z_0| \leq n^{1/2 +\eps_1}$, and all $1 \leq j \leq k$, we have
\begin{equation}\label{lam1}
 |\nabla^{m} \lambda_{i_j}(A(z))| \ll n^{O(\eps_1)} n^{-m}
\end{equation}
and
\begin{equation}\label{lam2}
 n |\nabla^{m} P_{i_j,p_j,q_j}(A(z))| \ll n^{O(\eps_1)} n^{-m}
\end{equation}
and
\begin{equation}\label{lam3}
 |\nabla^{m} Q_{i_j}(A(z))| \ll n^{O(\eps_1)} n^{-m}
\end{equation}
for all $z$ with $|z| \leq n^{1/2 +\eps_1}$ and all $0 \leq m \leq 10$.
\end{lemma}

The bounds \eqref{lam1}, \eqref{lam3} were already proven in \cite[Lemma 59]{TVuniv}; the arguments in the proof of that lemma also show that
$$ Q_{i_j}(A(z)) \ll n^{\eps_1}$$
uniformly for all $z$ with $|z| \leq n^{1/2+\eps_1}$.

Now we prove \eqref{lam2}.  Arguing inductively as in the proof of \cite[Lemma 59]{TVuniv}, it suffices to establish the claim for $z$ in the ball $B(z_0,n^{-1-2\eps_1})$.  Let us first establish this for $z$ whose real and imaginary parts are a multiple of $n^{-C_1}$.  At this value of $z$, we can apply Proposition \ref{better-cor} exactly as in \cite[Section 4.3]{TVuniv} to obtain the bounds
\begin{align}
\| P_{i_j}(\nabla^m P_{i_j}) P_{i_j} \|_F &\ll n^{-m+O(\eps_1)} (\sum_{0 \leq \alpha \leq \log n} \frac{2^\alpha}{r_\alpha})^m \label{extra-1a}\\
\| P_{i_j}^{(\alpha)}(\nabla^m P_{i_j}) P_{i_j} \|_F =
\| P_{i_j}(\nabla^m P_{i_j}) P_{i_j,\alpha} \|_F
 &\ll \frac{2^{\alpha/2}}{r_\alpha} n^{-m+O(\eps_1)} (\sum_{0 \leq \alpha \leq \log n} \frac{2^\alpha}{r_\alpha})^{m-1}
 \label{extra-2a} \\
\| P_{i_j}^{(\alpha)}(\nabla^m P_{i_j}) P_{i_j}^{(\beta)} \|_F
 &\ll \frac{2^{\alpha/2} 2^{\beta/2}}{r_\alpha r_\beta} n^{-m+O(\eps_1)} (\sum_{0 \leq \alpha \leq \log n} \frac{2^\alpha}{r_\alpha})^{m-2}
\label{extra-3a}
\end{align}
for all $0 \leq m \leq 10$, $1 \leq j \leq k$, and $0 \leq \alpha,\beta \leq \log n$, where $r_\alpha$ is the minimal value of $|\lambda_i - \lambda_{i_j}|$ for $|i-i_j| \geq 2^\alpha$, and $P_{i_j}^{(\alpha)}$ is the spectral projection to those eigenvalues with $2^\alpha \leq |i-i_j| < 2^{\alpha+1}$.

In \cite[Section 4.3]{TVuniv} it is shown that
$$ \sum_{0 \leq \alpha \leq \log n} \frac{2^\alpha}{r_\alpha} \ll n^{O(\eps_1)};$$
in particular
$$ r_\alpha \gg n^{-O(\eps_1)} 2^\alpha.$$
We conclude that
\begin{equation}\label{sal}
\| P_{i_j}^{(\alpha)} (\nabla^m P_{i_j}) P_{i_j}^{(\beta)} \|_F
 \ll 2^{-\alpha/2} 2^{-\beta/2} n^{-m+O(\eps_1)}
\end{equation}
for all $-1 \leq \alpha,\beta \leq \log n$, where we adopt the convention that $P_{i_j}^{(-1)} := P_{i_j}$.
Now, we expand
\begin{align*}
(\nabla^m P_{i_j})_{p_j,q_j} &= e_{p_j}^* (\nabla^m P_{i_j}) e_{q_j} \\
&= \sum_{-1 \leq \alpha,\beta \leq \log n} (e_{p_j}^* P_{i_j}^{(\alpha)}) (P_{i_j}^{(\alpha)} (\nabla^m P_{i_j}) P_{i_j}^{(\beta)} ) (P_{i_j}^{(\beta)} e_{q_j}).
\end{align*}
Applying the triangle and Cauchy-Schwarz inequalities we conclude that
$$
|(\nabla^m P_{i_j})_{p_j,q_j}| \leq
\sum_{-1 \leq \alpha,\beta \leq \log n} \| e_{p_j}^* P_{i_j}^{(\alpha)}\| \|P_{i_j}^{(\alpha)} (\nabla^m P_{i_j}) P_{i_j}^{(\beta)} \|_F \|P_{i_j}^{(\beta)} e_{q_j}\|.$$
From the hypothesis \eqref{uimenj} and Pythagoras' theorem we have
$$ \| e_{p_j}^* P_{i_j}^{(\alpha)}\| \ll n^{\eps_1} 2^{\alpha/2} n^{-1/2}$$
and similarly
$$ \| P_{i_j}^{(\beta)} e_{q_j} \| \ll n^{\eps_1} 2^{\beta/2} n^{-1/2}$$
so from \eqref{sal} we obtain \eqref{lam2} as required, at least when the real and imaginary parts of $z$ are multiples of $n^{-C_1}$.  The general case can then be handled (for $C_1$ large enough) by an appeal to Lemma \ref{crude} by arguing exactly as in \cite[Section 4.3]{TVuniv}.

\section{Proof of Theorem \ref{theorem:CLT2}}\label{clt2-sec}

We now prove Theorem \ref{theorem:CLT2}.  The main tool is the following general central limit theorem:

\begin{proposition}[Central limit theorem]\label{propo1}  Suppose one has a random unit vector $u = (u_1,\ldots,u_n) \in \R^n$ such that
\begin{enumerate}
\item[(i)] (Exchangeability) The distribution of $u$ is symmetric with respect to the permutation group $S_n$.
\item[(ii)] For any distinct $i_1,\ldots,i_k$ with $k$ fixed, $\sqrt{n} u_{i_1},\ldots, \sqrt{n} u_{i_k}$ converges jointly in distribution to $k$ iid copies of $N(0,1)$ as $n \to \infty$.
\item[(iii)] (Symmetry) The distribution of $u$ is symmetric with respect to the reflection group $\{-1,+1\}^n$.
\end{enumerate}
Then for any unit vector $a  \in \R^n$ (depending on $n$, of course), $\sqrt{n} a \cdot u$ converges in distribution to $N(0,1)_\R$.
\end{proposition}

Indeed, part (b) of Theorem \ref{theorem:CLT2} follows immediately from this proposition and Corollary \ref{goegue}, noting from the symmetries of $M_n$ that the eigenvector $u_i(M_n)$ is both symmetric and exchangeable.  Part (a) also follows after first reducing to the case $a \cdot e_1 = 0$ (noting from Corollary \ref{goegue} that $u_i(M_n) \cdot e_1$ converges in distribution to $|N(0,1)_\R|$, so that $(u_i(M_n) \cdot e_1) (a \cdot e_1)$ converges in distribution to $0$ if $a \cdot e_1 = o(1)$).

We now prove Proposition \ref{propo1}.  Let $A = A_n$ be a quantity growing slowly to infinity that we will choose later.  We truncate
$$ u_i = u_{i,\leq} + u_{i,>}$$
where $u_{i,\leq} := u_i 1_{|u_i| \leq A/\sqrt{n}}$ and $u_{i,>} := u_i 1_{|u_i| > A/\sqrt{n}}$.  We then split $u = u_{\leq} + u_>$ correspondingly.  It will suffice to show that, for a suitable choice of $A$,
\begin{enumerate}
\item $a \cdot u_{\leq}$ converges in distribution to $N(0,1)_\R$, and
\item $a \cdot u_>$ converges in probability to zero.
\end{enumerate}

We first consider the second claim.    Here we use the second moment method.  It suffices to show that
$$ n \E |a \cdot u_>|^2 = o(1).$$
Set $a=(a_1, \dot, a_n)$. The left-hand side can be expanded as
$$ n \sum_{i=1}^n \sum_{j=1}^n a_i a_j \E u_{i,>} u_{j,>}.$$
Using the symmetry hypothesis (iv), we see that if $i \neq j$, then $u_{i,>} u_{j,>}$ has a symmetric distribution and thus has mean zero.  Thus only the diagonal terms contribute.  Using hypothesis (i) and the unit normalization of $a$, the second moment becomes
$$ n \E u_{1,>}^2.$$
On the other hand, from (i) we have
$$ n \E u_1^2 = 1$$
while from (ii) we have
$$ n \E u_1^2 1_{|u_1| \leq K} = \E G^2 1_{|G| \leq K} + o(1)$$
for any fixed $K$, where $G \equiv N(0,1)$, and thus (since $A = A_n$ goes to infinity)
$$ \liminf_{n \to \infty} n \E u_1^2 1_{|u_1| \leq A} \geq \sup_K \E G^2 1_{|G| \leq K} = 1.$$
We conclude that
$$ n \E u_{1,\leq}^2 = 1 - o(1)$$
and
$$ n \E u_{1,>}^2 = o(1)$$
and Claim 2 follows.

Now we turn to Claim 1.  Here we use the moment method. By Carleman's theorem (see e.g. \cite{BS}), it suffices to show that for each fixed positive integer $k$,
$$ n^{k/2} \E (a \cdot u_{\leq})^k = \E G^k + o(1).$$
The left-hand side expands as
$$ n^{k/2} \sum_{1 \leq i_1,\ldots,i_k \leq n} a_{i_1} \ldots a_{i_k} \E u_{i_1,\leq} \ldots u_{i_k,\leq}.$$
By the symmetry hypothesis (iii), the expectation vanishes unless each index $i$ appears an even number of times.  Using hypothesis (ii) (and (iii)), we see that
$$
n^{k/2} \E u_{i_1,\leq} \ldots u_{i_k,\leq} = \E G_{i_1} \ldots G_{i_k} + o(1)$$
where $G_1,\ldots,G_n$ are iid copies of $N(0,1)$, uniformly in $i_1,\ldots,i_k$, if $A_n$ grows sufficiently slowly to infinity.  Observe that
$$ \sum_{1 \leq i_1,\ldots,i_k \leq n} a_{i_1} \ldots a_{i_k} \E G_{i_1} \ldots G_{i_k} = \E \left(\sum_{i=1}^n a_i G_i\right)^k = \E G^k$$
since $\sum_{i=1}^n a_i G_i \equiv G$, and that (as before) the summands vanish unless each index $i$ appears an even number of times.  Thus it suffices to show that
$$ \sum_* a_{i_1} \ldots a_{i_k} = O(1)$$
where the sum $*$ is over all $k$-tuples $1 \leq i_1,\ldots,i_k \leq n$ in which each index $i$ appears an even number of times, and the implied constants in the $O()$ notation are allowed to depend on $k$.

Suppose there are $l$ distinct indices appearing, then the contribution of this case is at most
$$ O\left( \sum_{j_1=1}^n \ldots \sum_{j_l=1}^n a_{j_1}^2 \ldots a_{j_l}^2 \right)$$
(since we have $a_{j_r}^m \leq a_{j_r}^2$ whenever $m$ is a positive even number).  But as $a$ is a unit vector, this sums to $O(1)$ as required.  This concludes the proof of Proposition \ref{propo1} and hence Theorem \ref{theorem:CLT2}.

\begin{remark}
An extension of the above argument shows that, under the hypotheses of Proposition \ref{propo1}, if $a^1,\ldots,a^l$ are a bounded number of orthonormal vectors in $\R^n$, then $\sqrt{n} a^1 \cdot u, \ldots, \sqrt{n} a^l \cdot u$ converges jointly in distribution to $l$ iid copies of $N(0,1)_\R$.
\end{remark}

\section{Proof of Theorem \ref{green}}\label{green-sec}

We now prove Theorem \ref{green}.  Let $M_n, M'_n, C, z, E, \eta,p,q$ be as in that theorem. Let $c_1>0$ be a small constant to be chosen later, and let $c_2 > 0$ be an even smaller constant (depending on $c_1$) to be chosen later.  Write $A_n := \sqrt{n} M_n$ and $A'_n :=\sqrt{n} M'_n$.   

Let us call an expression depending on $M_n$ (or $A_n$) \emph{stable} if it only changes by $O(n^{-c})$ for some $c>0$ if $M_n$ is replaced by $M'_n$.  Our task is thus to show that $\E G\left( \left(\frac{1}{\sqrt{n}} M_n-zI\right)^{-1}_{pq} \right)$ is stable.

We first subtract off the ``global'' portion of the resolvent $\left(\frac{1}{\sqrt{n}} M_n-zI\right)^{-1}_{pq}$.  Set $z_0 := E + i n^{-1+c_2/2}$.  Applying the local semicircle law from \cite[Theorem 2.1]{EYY1}, one has\footnote{Strictly speaking, the statement of \cite[Theorem 2.1]{EYY1} only controls the diagonal component $p=q$ of the resolvent.  However, an inspection of the proof of that theorem (see in particular \cite[(3.13)]{EYY1} and \cite[Proposition 3.3]{EYY1}) reveals that the off-diagonal components $p \neq q$ were also controlled by the argument.}
$$ \left(\frac{1}{\sqrt{n}} M_n-z_0I\right)^{-1}_{pq} = m_{sc}(z_0) \delta_{pq} + O(n^{-c})$$
with overwhelming probability for some $c>0$, where $\delta_{pq}$ is the Kronecker delta function and $m_{sc}(z_0)$ is semicircular Stieltjes transform
$$ m_{sc}(z_0) := \frac{-z_0+\sqrt{z_0^2-4}}{2}.$$
This type of result already gives the claim when $\eta \geq n^{-1+c_2/2}$, so we may assume that $\eta < n^{-1+c_2/2}$.
After shifting $G$ by $m_{sc}(z_0) \delta_{pq}$, and using the regularity bounds on $G$, it thus suffices to show that the expression
$$ \E G\left( \left(\frac{1}{\sqrt{n}} M_n-zI\right)^{-1}_{pq} - \left(\frac{1}{\sqrt{n}} M_n-z_0 I\right)^{-1}_{pq} \right)$$
is stable.

From \eqref{green-eq} we have
$$
\left(\frac{1}{\sqrt{n}} M_n-zI\right)^{-1}_{pq} - \left(\frac{1}{\sqrt{n}} M_n-z_0 I\right)^{-1}_{pq}
 = \sum_{i=1}^n F(\lambda_i(A_n)) n P_{i,p,q}(A_n)$$
 where
$$ F(x) := \frac{1}{x - nz} - \frac{1}{x - nz_0}.$$
Note that $F$ decays quadratically in $x$ rather than linearly, due to the subtraction of the comparison term $\frac{1}{x-nz_0}$; this will allow us to easily neglect the contribution of the spectrum that is far from $E$ in the rest of the argument.

We now invoke the eigenvalue rigidity result from \cite[Theorem 2.2]{EYY-rigid}.  Among other things, this theorem gives an interval $[i_-,i_+] \subset [1,n]$ of length $i_+ - i_- = O(n^{c_2})$ (depending only on $n$, $E$, and $c_2$) which contains most of the eigenvalues close to $E$, in the sense that we have
$$ \inf_{i \in [1,n] \backslash [i_-,i_+]} |\lambda_i(A_n) - E| \geq n^{c_2-o(1)}$$
with overwhelming probability.  In fact, we have the stronger assertion that
$$ |\lambda_i(A_n) - E| > n^{-o(1)} (n^{c_2} + \operatorname{dist}(i, [i_-,i_+]))$$
for all $i \in [1,n] \backslash [i_-,i_+]$ with overwhelming probability.  In particular, this implies that
$$
 |F( \lambda_i(A_n) )| \leq \frac{n^{c_2/2 + o(1)}}{(n^{c_2} + \operatorname{dist}(i, [i_-,i_+]))^2}.
$$
and thus
\begin{equation}\label{soo} 
\sum_{i \in [1,n] \backslash [i_-,i_+]} |F( \lambda_i(A_n) )| \leq n^{-c_2/2 + o(1)}
\end{equation}
with overwhelming probability.

Also, by the delocalization of eigenvalues (established in the bulk in \cite[Theorem 4.8]{ESY3} (see also the earlier result \cite[Theorem 1.2]{ESY} handling the smooth case), and up to the edge in \cite[Proposition 1.12]{TVedge}), we have
$$ |u_{i,p}(A_n)| = O(n^{-1/2+o(1)})$$
with overwhelming probability for all $1 \leq i,p \leq n$, and hence
$$ n P_{i,p,q}(A_n) = O(n^{o(1)})$$
with overwhelming probability for all $1 \leq i,p,q, \leq n$.  As such, we see that with overwhelming probability, we have
\begin{align*}
\left(\frac{1}{\sqrt{n}} M_n-zI\right)^{-1}_{pq} & - \left(\frac{1}{\sqrt{n}} M_n-z_0 I\right)^{-1}_{pq} \\
&= \sum_{i \in [i_-,i_+]} F(\lambda_i(A_n)) n P_{i,p,q}(A_n) + O( n^{-c_2/2+o(1)} )
\end{align*}
with overwhelming probability.  Using the regularity bounds on $G$, it thus suffices to show that the quantity
$$ \E G( \sum_{i \in [i_-,i_+]} F(\lambda_i(A_n)) n P_{i,p,q}(A_n) )$$
is stable.

We let $\chi: \C \to [0,1]$ be a cutoff function supported on $B(0,1)$ that equals $1$ on $B(0,1/2)$.  Define the truncated version
$$ \tilde F(x) := F(x) (1 - \chi(n^{c_1} (x-z)))$$
of $F$.  From the hypothesis \eqref{mo}, we see that for each $i \in [i_-,i_+]$, one has
$$ F(\lambda_i(A_n)) = \tilde F(\lambda_i(A_n))$$
with probability at least $1-O(n^{-10c_2})$ (say), if $c_2$ is sufficiently small depending on $c_1$ and on the implied constant in the high probability event \eqref{mo}.  In particular, by the union bound, we have
$$ G( \sum_{i \in [i_-,i_+]} F(\lambda_i(A_n)) n P_{i,p,q}(A_n) ) = G( \sum_{i \in [i_-,i_+]} \tilde F(\lambda_i(A_n)) n P_{i,p,q}(A_n) )$$
with high probability. 

 Similarly for $A'_n$ using \eqref{no}.  It thus suffices to show that the quantity
$$ \E G( \sum_{i \in [i_-,i_+]} \tilde F(\lambda_i(A_n)) n P_{i,p,q}(A_n) )$$
is stable.  But this follows directly from Theorem \ref{theorem:main0} (if $c_1, c_2$ are small enough).  The proof of Theorem \ref{green} is now complete.

\begin{remark} 
The exponential decay hypothesis (Condition {\condo}) in Theorem \ref{green} is needed only to be able to use the eigenvalue rigidity result from \cite{EYY-rigid} and the local semicircle law from \cite{EYY1}.  It is quite likely that in both cases, one can relax Condition {\condo} to Condition {\condone} (conceding some factors of $n^{O(1/C_0)}$ in the process), which would then allow one to achieve a similar relaxation in Theorem \ref{green}.
\end{remark}

\end{document}